\documentclass[12pt,a4paper]{article}
\usepackage[latin1]{inputenc}
\usepackage{enumerate}
\usepackage{amssymb}
\usepackage{amsthm}
\usepackage{dsfont}
\usepackage{graphicx}
\usepackage{amsthm}
\usepackage{epsfig}
\usepackage[intlimits]{amsmath}
\usepackage{paralist}
\usepackage{mathrsfs}
\usepackage{faktor}
\usepackage{empheq}
\usepackage{fancyhdr}
\usepackage{hyperref}
\usepackage{makeidx}

\usepackage[all]{xy}
\usepackage[english]{babel}

\theoremstyle{definition}
\newtheorem{definition}{Definition}[section]

\newtheorem{dfn}[definition]{Definition}
\newtheorem{rmk}[definition]{Remark}
\theoremstyle{plain}
\newtheorem{satz}[definition]{Proposition}
\newtheorem{thm}[definition]{Theorem}
\newtheorem{lem}[definition]{Lemma}
\newtheorem{cor}[definition]{Corollary}

\newcommand{\R}{\ensuremath{\mathbb{R}}}	
\newcommand{\C}{\ensuremath{\mathbb{C}}}

\hypersetup{bookmarks=true}

\makeindex

\begin{document}
\author{Jonas Grabbe and Andrei Moroianu}
\title{On conformal holonomy}
\begin{center}
~\\
\vspace{2cm}{\LARGE  On the holonomy groups of Weyl manifolds} \\
\vspace{0.5cm}{Jonas Grabbe}
\end{center}
\vspace{0.5cm}
\begin{abstract}
We classify the possible local holonomy groups of Weyl connections. The Berger-Simons theorem and the Merkulov-Schwachh\"ofer classification of holonomy groups of irreducible torsion-free connections leaves us with the remaining case, where the Weyl connection $D$ is reducible and non-closed. In this case, it was shown in \cite{BeMo} that the Weyl structure is an adapted Weyl structure of a non-closed conformal product. Furthermore we prove that non-closed Einstein-Weyl product structures only exist in dimension $4$.   
\end{abstract}
\section{Introduction}

In \cite{ BeMo}, the notion of a conformal product was introduced, by generalizing a property characterizing the Riemannian product of two Riemannian manifolds, namely, the existence of two complementary orthogonal Riemannian submersions. Thus a conformal structure on a manifold is said to be a \em conformal product \em if it admits two conformal submersions with orthogonal fibres intersecting transversally. \\

A conformal structure is a positive definite symmetric bilinear tensor with values in the square of the weight bundle. In conformal geometry, the role of the Levi-Civita connection is played by an affine space of torsion-free connections preserving the conformal structure (called \em Weyl structures\em), which is in one-to-one correspondence with the space of connections on the weight bundle. In particular, Weyl structures coincide locally with Levi-Civita connections of Riemannian metrics in the conformal class if and only if the corresponding connection on the weight bundle has vanishing curvature. These Weyl structures are called closed, thus we will restrict our study on the non-closed Weyl structures.\\ 

In 1918 Hermann Weyl introduced conformal structures in his attempt to formulate a unified field theory. He thought that the conformal structure was able to unify gravity and electromagnetic interaction. His physical motivation was, that the universe is not really a Riemannian manifold, for there is no absolute measure of length, oppose to Einstein's model for physical space. 
In the theory he conjectured a determination of length at one point induces only a first-order approximation to determination of length at surrounding points. Einstein's 
counter argument was that no stable frequency of atomic clocks could be expected. In that case there can not be physics, because everybody would have his own laws, and there would be chaos \cite{Chan}. \\
Although Weyl's theory failed for physical reasons, it remains a beautiful piece of mathematics.  
   
The work \cite{BeMo} was partially inspired by a the paper \cite{Mo}, where  Spin conformal manifolds with Weyl structures $D$ carrying parallel spinors have been studied. It turned out, that the spin holonomy representation of a non-closed Weyl structure has no fixed point, except in dimension $4$, where  local examples exist. In \cite{BeMo} the analogous question was considered for exterior forms. Here we conclude this study by giving also an analogous result.
\\
In this paper we classify the possible local holonomy groups of Weyl connections. Holonomy groups are defined by parallel displacements along loops in $x \in M$ with respect to a connection form on a principal fibre bundle or with respect to a covariant derivative in a vector bundle over $M$. 
The Riemannian case is well studied and classified by the famous Berger-Simons theorem \cite{Besse}. Since conformal geometry of a closed Weyl structure is locally equivalent to Riemann geometry, in the case of a closed Weyl structure, the Berger-Simons theorem applies. On the other hand if the Weyl structure is irreducible, the Merkulov-Schwachh\"ofer classification of possible holonomies of irreducible torsion-free connections \cite{MeSc} applies. Hence the remaining case to study is the case where the Weyl structure is reducible and non-closed. In this case, it was shown in \cite{BeMo} that the Weyl structure is an adapted Weyl structure of a non-closed conformal product. Our main result is the following:
\begin{thm}
A non-closed conformal product $(M_1 \times M_2, c , D )$, with $n_i := dim(M_i)$
has restricted holonomy 
$$
 \R_{ + } ^{ \ast } \times SO(n_1) \times SO(n_2),
$$
except if $n_1 =n_2=2$, when the restricted holonomy group is either $ \R_{ + } ^{ \ast } \times SO(2) \times SO(2)$ or $\C ^{ \ast }$. 
\end{thm} 
Finally, in section $4$ we will discuss Einstein-Weyl products, which are reducible Weyl manifolds whose trace-free symmetric part of the Ricci tensor vanishes. In this last section we prove that there are no non-closed Einstein-Weyl product structures $(M,c,D)$ except when $n=4$ by generalizing a proposition in \cite{BeMo}.
\begin{satz}
A non-closed conformal product $(M, c , D )$ with $M=M^{n_1}_1 \times M^{n_2}_2$, $c=[g_1+e^{2f}g_2]$ and $D$ the adapted Weyl structure is Einstein-Weyl if and only if it is locally isomorphic to a conformal product $(M_1 \times M_2 ,c= [g_1 +e^{2f} g_2],D) $, where $M_1 $ and $M_2 $ are open sets of $\R ^2 $, $g_i$ is the flat metric on $ M_i$ and the function $f: M_1 \times M_2 \subset \R ^4 \to \R $ satisfies the $Toda$-type equation 
$$
e^{2f}(\partial _{11} f + \partial _{22} f) + \partial _{33} f + \partial
 _{44} f =0.
$$ 
\end{satz}

\section{Preliminaries}
For the theory of Weyl derivatives it is convenient to work with densities.
\begin{dfn}
Let $V$ be a real $n$-dimensional vector space and $k$ a real number. A homogeneous map $\mu :\Lambda ^n V \backslash \{ 0 \} \to \R$ with the property $\mu(\lambda \omega )= |\lambda |^{-k/n} \mu(\omega )$ for all $\lambda \in \R ^{\ast}$ is called a \em density of weight \em $k$ or a $k$-\em density \em. 
\end{dfn}
The set of all densities of weight $k$ forms a one dimensional vector space 
$L^k(V)$ or simply $L^k$.
This vector space is oriented, since a non-trivial density takes either positive or negative values. Furthermore $L^k$ naturally carries the representation $\lambda.\mu = |\lambda |^k \mu $ of the center of the group $GL(V)$ or equivalently the representation $A.\mu= |det A|^{k/n}\mu$ of $GL(V)$. \\ 
Let $M^n$ be a $n$-dimensional manifold, then the \em density line bundle \em $L^k=L^k(TM)$ of $M$ is defined to be the bundle whose fibre at $x\in M$ is $L^k(T_xM)$ or equivalently to be the associated bundle:
$$
L^k = GL(M)\times _{| det|^{k/n}} \R.
$$     
We note, that $L^0 = \R$. If $k$ is a natural number, then $L^k$ is the $k$-tensor product of $L^1=:L.$ The dual of $L$ is isomorphic to $L^{-1}$. Moreover we have the following isomorphisms $L^{k_1} \otimes L^{k_2} \simeq L^{k_1 +k_2}$, $(L^k)^p \simeq (L^p)^k$ and  if $M$ is oriented $L^{-n}\simeq |\Lambda ^n(T^{\ast}M)| = \delta M$, which is the bundle of \em densities \em on $M$.
Elements of $L^1$ may be thought of as scalars with dimension of $(length)$. In general the tensor bundle $TM^{\otimes i} \otimes T^{\ast} M^{\otimes j}\otimes L^k$ and any subbundle, quotient bundle, section or element will be said to have weight $i-j+k$, or dimension of $(length)^{i-j+k}$. \\
The real line bundles $L^k$ are trivializable since they are orientable. However, there is no preferred orientation on $L^k$, except for $k=0$.
\begin{dfn}
A non vanishing, usually positive section of $L^k$, for $k \neq 0$
will be called a \em length scale \em or \em gauge \em of weight $k$ or with dimension of $(length)^k$.  
\end{dfn}   
Real numbers $\R=L^0$ are weightless and dimensionless. Vectors of $T_xM$, $x \in M$ have weight $+1$ and describe translations of dimension $(length)$. The sections of $L^{-n}$ over $M$ play the role of natural integrands. The tensor product $\Lambda ^n TM \otimes L^{-n}$ is the weightless space of pseudoscalars. This one dimensional space naturally carries a norm given by $|\mu \otimes \omega |:= |\mu (\omega )|$. The two orientations of $T_xM$, $x \in M$ are in one-to-one correspondence with the two unit elements of $\Lambda ^n TM \otimes L^{-n}$. 

A conformal structure on a smooth manifold $M$ is an equivalent class $c$ of Riemannian metrics, where two Riemannian metrics $g,\tilde{g} \in C^{\infty }(M,S^2(T^{\ast }M)) $ are equivalent if $\tilde{g}=e^{2f}g$ for a smooth function $f$ on M. In conformal geometry it is interesting to not just look at a conformal structure as equivalence class but rather like an algebraic structure in the following sense:
\begin{dfn} 
A \em conformal structure \em on $M$ is a symmetric positive definite bilinear form $c$ on $TM \otimes L^{-1}$, or equivalently, a symmetric positive definite bilinear form on $TM$ with values in $L^2$.
\end{dfn}
A conformal structure can also be seen as a reduction $CO(M)$ of $GL(M)$ to the conformal group $CO(n) \simeq \R ^+ \times O(n) \subset GL(n, \R)$, where the isomorphism is obtained by identifying the positive real numbers $\R^+$ with the subgroup of dilatations. \\
More precisely it is a section $c \in C^{\infty }(M,S^2(T^{\ast }M)\otimes L^2)$ which is everywhere positive definite. Hence once a conformal structure $c$ fixed, there is a one-to-one correspondence between positive sections $l$ of $L$ and Riemannian metrics on $M$ : $c = g\otimes l^2. $  

In conformal geometry, the role of the Levi-Civita connection in Riemannian geometry is played by the affine space of Weyl connections, which are torsion-free and preserve the conformal structure.   
\begin{dfn}
A \em Weyl connection \em on a conformal manifold is torsion-free connection on $CO(M)$, or equivalently a torsion-free connection on $GL(M)$ induced by an covariant derivative $D$ on the bundle $S^2(T^{\ast}M)\otimes L^2$ which preserves the conformal structure, i.e. $Dc=0$. 
\end{dfn}
The existence and uniqueness of the Levi-Civita connections is a special case of the following central result, which is the fundamental theorem in conformal geometry and was proven by H.Weyl in \cite{Weyl}.
\begin{thm}
There is a one-to-one correspondence between Weyl connections of $TM$ and covariant derivatives on $L$, induced by restriction to the application which associates to each linear connection $D$ on $TM$ a connection $\nabla ^D$ on $L$.
\end{thm}
\begin{proof}
Every connection on $CO(M)$ induces covariant derivatives $D$ on $TM$ and $\nabla ^D$ on $L$. 
A Weyl connection is characterized by being torsion-free, i.e. $T(X,Y):=D_X Y-D_Y X +[X,Y]=0$ and by preserving the conformal structure, i.e. $Dc=0$. 
Then as in Riemannian geometry, these two relations are equivalent to the generalized $Koszul$ $formula$:
\begin{align*}
2c(D_X Y , Z)  = & \nabla ^D _X(c(Y,Z)) +  \nabla ^D _Y(c(X,Z)) -  \nabla ^D _Z(c(X,Y)) \\ & + c([X,Y],Z) +c([Z,X],Y) -c([Y,Z],X),
\end{align*}  
for all vector fields $X,Y,Z$ on $M$. Hence, every covariant derivative $\nabla ^D$ on $L$ induces by the formula a covariant derivative on $TM$, and thus  on $CO(M)$, which is clearly torsion-free and preserves the conformal structure.
\end{proof}
A conformal manifold $(M,c)$ equipped with a Weyl derivative $D$ is called a \em Weyl manifold \em $(M,c,D)$.
Henceforth we denote with $D$ the Weyl connection and its associated linear connection on $L$. \\
Sometimes it is useful to compare two Weyl derivatives on a conformal manifold. Consider the Weyl connection $D$ and $D'$ on $(M,c)$ seen as a covariant derivative on $L$. The difference $D-D'$ is a $1$-form on $TM$ with values in $End (L) = L^{\ast } \otimes L = \R $. Hence there exists a real $1$-form $\theta \in \Omega ^1 (M,\R ) $, such that 
\begin{align} \label{Dl-Dl} 
D'l=Dl + \theta \otimes l \ \ \ \forall l \in C^{\infty } (L). 
\end{align}     
These Weyl connections seen as connections on $TM$ have a $\mathfrak{co} (TM)$ valued $1$-form $\Gamma $ as their difference, i.e. $D'-D =\Gamma$.  
With (\ref{Dl-Dl}) and the Koszul formula we can express $\Gamma $ in terms of the $1$-form $\theta$:
$$
D'_X Y = D_X Y + \theta (X) Y + \theta (Y) X - c(X,Y)\theta^{\sharp},
$$   
for all vector fields $X,Y$ on $M$. Here $\theta^{\sharp}$ is the section of $TM \otimes L^{-2}$ defined by $\theta  (X) = c(\theta^{\sharp} , X)$ for all $X \in TM$. 
Since for all Riemannian metrics $g\in c$, the Levi-Civita connection $D^g$ is a Weyl connection, for all Weyl connections $D$ and all $g \in c$ one can write
\begin{align} \label{Dl-Dl2} 
Dl=D^gl + \theta _g \otimes l \ \ \ \forall l  \in C^{\infty } (L).
\end{align}
The $1$-form $\theta _g$ is called $Lee$ $form$ of $D$ with respect to $g$. The gauge $l_g$ corresponding to $g=c\otimes l_g ^{-2}$ is $D^g$-parallel. This follows, since $D^g$ seen as the Levi-Civita connection preserves the metric $g$, i.e. $D^g g=0$ and $D^g$ seen as Weyl connection preserves the conformal structure $c$, i.e. $D^g c=0$: 
$$
0=D^g g= D^g (c \otimes l^{-2}_g) = D^g c \otimes  l_g^{-2} + c\otimes D^g  l_g^{-2}, 
$$
i.e. $ D^g l_g =0$. Now we see, that $\theta _g$ is the connection form of $D$ on $L$ with respect to the gauge $l_g$:
$$
D_X l_g = \theta _g (X) l_g \ \ \ \forall X \in TM,
$$
equivalently $D_X g = -2\theta _g(X) g$ for all vector fields $X$ on $M$. Hence we will call Weyl connections $closed$ (resp. $exact$ ) if the Lee form $\theta _g$ is closed (resp. exact) for all $g\in c$. \\
Furthermore we can establish the relation between the two Lee forms of a Weyl connection with respect to two Riemannian metrics in $c$.
Let us consider the metrics $g=c\otimes l_g^{-2}$ and $\tilde{g}=c\otimes l_{\tilde{g}}^{-2}$ in $c$, then there exists a function $f \in C^{\infty}(M)$, such that $\tilde{g} =e^{2f}g$. First we see that the length scales differ by $e^f$, i.e. $l_{\tilde{g}}=e^{-f}l_g$. Then, for the respective Lee forms we derive that 
$$
\theta _{\tilde{g}} \otimes l_{\tilde{g}} =Dl_{\tilde{g}} =D (e^{-f}l_g)=-df e^{-f} l_g + e^{-f}Dl_g= -df e^{-f} l_g + e^{-f}\theta _g \otimes l_g =(\theta _g -df)\otimes l_{\tilde{g}},
$$ 
i.e. $\theta _{\tilde{g}}= \theta _g -df $. \\
The curvature $F^D$ of a Weyl connection seen as a covariant derivative on $L$ is a real $2$-form on $TM$. Let us fix a gauge $l$ and consider its corresponding metric $g\in c$, then 
\begin{align*}
F^D(X,Y)l&= D_X(D_Y l) -D_Y(D_X l) -D_{[X,Y]}l \\
&= D_X(\theta _g (Y)l) -D_Y(\theta _g (X) l) -\theta _g ([X,Y])l \\
&= X\cdot \theta _g (Y)l +\theta _g (Y)\theta _g (X)l -Y\cdot \theta _g (X) l - \theta _g (X)\theta _g (Y)l -\theta _g ([X,Y])l \\
&= X\cdot \theta _g (Y)l  -Y\cdot \theta _g (X) l-\theta _g ([X,Y])l \\
&= d\theta _g (X,Y)l.		
\end{align*}   
One notes that the closed $2$-form $F^D$ does not depend on the metric $g$ defining $\theta$.
\begin{rmk}
The original gauge theory of \em metrical relations \em introduced by  Weyl in \cite{Weyl} is the theory of Weyl derivatives, which is a gauge theory with gauge group $\R ^+$. This theory is a geometrization of classical electromagnetism. Originally Weyl interpreted a Weyl derivative as electromagnetic potential and its curvature as electromagnetic field, which then automatically satisfies the first Maxwell equation,
($dF^D=0$). 
However, as a model for electromagnetism, the gauge theory of metrical relations was subsequently rejected in favour of a $U(1)$ gauge theory.
\end{rmk}
 For Weyl's first attempts of geometrization, we will call the closed $2$-form $F^D$ \em Faraday form\em.
\begin{satz}
A Weyl structure $D$ is closed, respectively exact, if and only if $D$ is flat seen as a linear connection on $L$, i.e. $F^D=0$, respectively if $D$ admits a $D$-parallel global section.
\end{satz} 
\begin{proof}
Let $l_g \neq 0 $ be a $D$-parallel section, i.e.
$
0=Dl_g = \theta _g \otimes l_g,
$
hence $\theta _g = 0$ and $\theta _{\tilde{g}} =df$, for all $\tilde{g} = e^{2f}g$ in $c$.
\end{proof}
\begin{cor}
A Weyl connection $D$ is locally, respectively globally the Levi-Civita connection of a metric in the conformal class if and only if $D$ is closed, respectively exact.   
\end{cor}
\begin{proof}
Since a closed form is locally exact, we consider that $D$ is exact. Above we showed that $D$ is exact iff $\theta _g=0$ for a metric $g$ in $c$. Thus from (\ref{Dl-Dl2}) we obtain, that $D=D^g$, which is the Levi-Civita connection of $g$.
\end{proof}
This is why Riemannian geometry can be seen as special case of conformal geometry, i.e. Riemannian geometry is the geometry of an exact Weyl derivative on a conformal manifold.\\
Consider the curvature $R^D$ of a Weyl connection $D$ action on $TM$,
$$
R^D _{X,Y}Z =[D_X, D_Y ] Z  -D_{[X,Y]}Z
$$
with $X,Y$ and $Z$ vector fields on $M$. In contrast to the Riemann case, $R^D$ is not symmetric by pairs, and with $D=D^g +\Gamma $ the curvature decomposes into
$$
R^D _{X,Y}Z = R^g _{X,Y}Z + [D^g_X, \Gamma _Y ]Z + [\Gamma _X , D^g _Y]Z+[\Gamma _X, \Gamma _Y]Z -\Gamma _{[X,Y]} Z 
$$
with $X,Y$ and $Z$ vector fields on $M$ and $R^g$ the curvature of $D^g$.
Regarding $R^D$ as a section of $T^{\ast }M ^{\otimes 4} \otimes L^2$ by the formula $R^D (X,Y,Z,T)= c(R^D _{X,Y}Z,T)$, one can calculate that the symmetry failure of $R^D$ is measured by the Faraday form $F$ of $D$,
\begin{gather}  
\begin{aligned} \label{RofD}
R^D(X,Y,V,W )-R^D(V,W,X,Y) &= (F(X) \wedge Y -F(Y) \wedge X)(V,W) \\
& \ \ \ + F(X,Y)c(V,W)-F(V,W)c(X,Y),
\end{aligned}
\end{gather}
where the wedge product is defined for a weighted endomorphism $A\in End(TM) \otimes L^k$, by:
$$
(A(X)\wedge Y)(Z,T):=c(A(X),Z)c(Y,T)- c(A(X),T)c(Y,Z).
$$ 
For later use, the Ricci tensor of a Weyl structure $D$ is defined by (\cite{Gaud}, although we use a different sign convention for the curvature tensor):
$$
Ric^D (X,Y):= \frac{1}{2} \sum ^n _{k=1}(g(R^D_{X,e_k} e_k,Y) -g(R^D_{X,e_k} Y,e_k)),
$$ 
where $g$ is an arbitrary metric in the conformal class and $\{e_k \}$ is a local $g$-orthogonal frame.
\section{The holonomy classification}
In this section we want to prove ou main result, but first we give a short introduction to conformal product structure and holonomy.
For this we consider a conformal map $f$ between  $(M,c)$ and $(N,c')$. We call $f$ a \em conformal submersion \em if its differential restricted to $(ker df)^{\bot}$ is a conformal isomorphism in every point. 
\begin{dfn}
A conformal structure on the manifold $M:= M_1 \times M_2$ is a \em conformal product structure \em of $(M_1,c_1)$ and $(M_2,c_2)$ if and only if the canonical submersion $p_1: M \to M_1$ and $p_2: M \to M_2$ are orthogonal conformal submersions.
\end{dfn}  
In \cite{BeMo} the following correspondence between conformal products and reducible holonomy was shown:
\begin{thm}
A conforma manifold $M$ has (local) conformal product structure if and only if it carries a Weyl structure with reducible holonomy. 
\end{thm}
This establishes the existence of a unique \em adapted Weyl \em structure $D$ of the conformal product $(M,c)$. \\
\begin{dfn}
A conformal product $(M,c) $ is called a \em closed \em conformal product if the adapted Weyl structure is closed. 
\end{dfn} 
 
We want to classify the possible (local) holonomy groups of Weyl connections. As mentioned in the introduction, in the case where $D$ is closed, the Berger-Simons theorem applies. Furthermore Merkulov-Schwachh\"ofer's classification \cite{MeSc} of possible holonomies of irreducible torsion-free connections leaves us with the case when $D$ is non-closed and irreducible. Hence, consider $D$ an adapted Weyl structure of a non-closed conformal product. 
\begin{thm}\label{thm}
A non-closed conformal product $(M_1 \times M_2, c , D )$, with $n_i := dim(M_i)$
has restricted holonomy 
$$
 \R_{ + } ^{ \ast } \times SO(n_1) \times SO(n_2),
$$
except if $n_1 =n_2=2$, when the restricted holonomy group is either $ \R_{ + } ^{ \ast } \times SO(2) \times SO(2)$ or $\C ^{ \ast }$. 
\end{thm} 
Here $\R_{ + } ^{ \ast } \times SO(n_1) \times SO(n_2)\subset CO^+(n_1+n_2)$ is the subgroup with embedding $(r,A,B) \mapsto r \begin{pmatrix} A & 0 \\ 0 & B 
\end{pmatrix} $ and $\C ^{ \ast } =  \R_{ + } ^{ \ast } \times SO(2) \subset CO^+(4) $ is the subgroup with diagonal embedding $$(r,A) \mapsto r \begin{pmatrix} A & 0 \\ 0 & A 
\end{pmatrix}. $$ 
Before proving the theorem we are going to recall some basics about holonomy (\cite{Baum}). 
Let $M^n$ be a smooth manifold and let $\mathcal{L} (M,x)$ denote the set of all piecewise smooth loops in a point $x \in M$. If $M$ is equipped with a linear connection $\nabla$ on the tangent bundle $TM$, we can parallel translate a tangent vector $X$ in $T_xM$ along any given piecewise smooth curve $\gamma :[0,1] \to M $, starting at $x=\gamma (0)$, i.e. with $\nabla $ we can link the tangent spaces in different points via a vector space isomorphism $\mathcal{P}_{\gamma (t) }$, which explains the terminology \em connection\em. 
The vector space isomorphism is called \em parallel displacement \em $\mathcal{P}_{\gamma (t)}$ and for any smooth curve $\gamma$,
\begin{align*}
\mathcal{P}_{\gamma (t)} :T_{\gamma (0) } & M \to T_{\gamma (t)} M \\
         & X \mapsto \mathcal{P}_{\gamma (t)} (X):= X(t),
\end{align*}
where $X(t)$ is a vector field along $\gamma $ satisfying the equation $\nabla _{\dot{\gamma } (t)} X(t) =0$ for all $t \in [0,1] \subset \R  $.  

The \em holonomy group \em of $\nabla$ with respect to the base point $x$ is the Lie group of all parallel displacements along piecewise smooth loops in $x$:
$$
Hol_x(M,\nabla ) := \{ \mathcal{P}_{\gamma (1)} | \gamma \in \mathcal{L} (M,x)  \} \subset GL(T_x M).
$$
If $M$ is connected then this Lie group depends on the base point $x$ only up to conjugation. The holonomy group is connected if $M$ is simply connected. The identity component $Hol_0 (M,\nabla)$ is called \em restricted holonomy group\em , and is the group generated by parallel displacements along homotopically trivial loops.
The \em holonomy algebra \em $\mathfrak{hol}_x(M,\nabla)$ is the Lie algebra of the restricted holonomy group.  
Both are given with their respective representation on the tangent space $T_xM$, which is usually identified with $\R ^n$. Thus we can consider the holonomy group as matrix group $Hol_x(M, \nabla) \subset GL(n,\R)$ up to conjugation. At different points in a connected component the holonomy groups are conjugated by an element in $GL(n,\R)$, which is given by the parallel displacement along a piecewise smooth curve joining theses points. Furthermore the holonomy group is closed if it acts irreducibly, which is not true in general. 

A frequently used theorem to calculate holonomy groups which will be crucial later is the $Ambrose$-$Singer$ $Theorem$, which states that for a manifold $M$ with linear connection $\nabla $, the holonomy algebra $\mathfrak{hol}_x(M,\nabla)$ is equal to 
$$
span \{ \mathcal{P}_{\gamma (t) } ^{-1} \circ R^{\nabla} (\mathcal{P}_{\gamma (t) }X,\mathcal{P}_{\gamma (t) } Y) \circ \mathcal{P}_{\gamma (t) } \  | \gamma (0) =x,\  X,Y \in T_x M \}.
$$  
In particular this implies the following inclusion:
 $$
span\{ R^{\nabla} (X,Y)| X,Y \in T_x M \} \subset \mathfrak{hol}_x(M,\nabla). 
$$
Throughout we are going to denote the holonomy group and algebra with $Hol(\nabla)$ and $\mathfrak{hol} (\nabla)$, when there are not any confusions about the manifold and the point.

In order to prove now Theorem \ref{thm} we first prove two straightforward algebraic facts. Consider the canonical isomorphism $\sharp : T^{\ast}M \to TM$ and its inverse $\flat :TM \to T^{\ast}M$, traditionally called "raising", respectively "lowering" of indices, induced by the metric. 
Throughout we will not make any distinction between raised, respectively lowered indices, nor between the holonomy group and the matrix group we get 
by choosing a basis at a point in its tangent space.

Let $V$ be an euclidean vector space with scalar product $\langle \cdot , \cdot \rangle$.
\begin{satz} 
For the commutator in $End(V)$ of $2$-forms $(X_i\wedge Y_i)$, $X_i,Y_i \in V$, $i =1,2$, seen as skew-symmetric endomorphisms of $V$ it holds that 
\begin{gather}  
\begin{aligned} \label{eq}
[X_1 \wedge Y_1 , X_2 \wedge Y_2] &= \langle X_1,X_2 \rangle Y_1 \wedge Y_2
+\langle Y_1,Y_2 \rangle X_1\wedge X_2 \\ 
 &\ \ \ - \langle Y_1,X_2 \rangle X_1\wedge
 Y_2- \langle X_1,Y_2 \rangle Y_1\wedge X_2.
\end{aligned}
\end{gather}
\begin{proof}
The identification of two-forms with skew-symmetric endomorphisms is given by
$$
(X\wedge Y)(Z)= \langle X,Z \rangle Y - \langle Y,Z \rangle X,
$$
with $X,Y,Z \in V$.
Hence, for $X_1,X_2,Y_1,Y_2, Z \in V$

\begin{align*}
[X_1 \wedge Y_1 , X_2 \wedge Y_2](Z) &= X_1 \wedge Y_1 (\langle X_2,Z \rangle Y_2 - \langle Y_2,Z \rangle X_2) \\  
&\ \ \  -X_2 \wedge Y_2 ( \langle X_1,Z \rangle Y_1 - \langle Y_1,Z \rangle X_1) \\
&= \langle X_2,Z \rangle \langle X_1,Y_2 \rangle Y_1 - \langle X_2,Z \rangle \langle Y_1,Y_2 \rangle X_1 \\
&\ \ \ - \langle Y_2,Z \rangle \langle X_1,X_2 \rangle Y_1 + \langle Y_2,Z \rangle \langle Y_1,X_2 \rangle X_1 \\
&\ \ \ - \langle X_1,Z\rangle \langle X_2,Y_1 \rangle Y_2 + \langle X_1,Z \rangle \langle Y_2,Y_1 \rangle X_2 \\
&\ \ \ + \langle Y_1,Z \rangle \langle X_2,X_1 \rangle Y_2 - \langle Y_1,Z \rangle \langle Y_2,X_1 \rangle X_2.
\end{align*}
Then after resummation, it follows
\begin{align*}
[X_1 \wedge Y_1 , X_2 \wedge Y_2](Z) &= \langle X_1,X_2\rangle (\langle Y_1,Z \rangle Y_2 - \langle Y_2,Z \rangle Y_1) \\
&\ \ \ +\langle Y_1,Y_2 \rangle (\langle X_1,Z \rangle X_2- \langle X_2,Z \rangle X_1) \\
&\ \ \  - \langle Y_1,X_2 \rangle (\langle X_1,Z \rangle Y_2-\langle Y_2,Z \rangle X_1) \\
&\ \ \ -\langle X_1,Y_2 \rangle (\langle Y_1,Z \rangle X_2-\langle X_2,Z \rangle Y_1) \\
&= \langle X_1,X_2 \rangle Y_1 \wedge Y_2 (Z)
+\langle Y_1,Y_2 \rangle X_1\wedge X_2(Z) \\ &\ \ \ -\langle Y_1,X_2 \rangle X_1\wedge
 Y_2(Z) -\langle X_1,Y_2 \rangle Y_1\wedge X_2(Z).
\end{align*}
\end{proof}
\end{satz}
This equality helps us now to prove the following algebraic lemma, which itself will be used to prove our main result. 
\begin{lem} \label{lem}
Let $V_1,V_2$ be two euclidean vector spaces and $F\in \Lambda ^2(V_1\oplus V_2) \simeq \mathfrak{so} (V_1 \oplus V_2) $  a non-identically zero two-form which vanishes over $V_i\wedge V_i$, $i=1,2$. Let $\mathfrak{g}\subset \Lambda ^2(V_1 \oplus V_2) $ be the Lie algebra generated by the Lie brackets of $F$ and two-forms $(X \wedge Y)\in V_1\otimes V_2 \subset \Lambda ^2(V_1\oplus V_2)$. Then 
$$
\mathfrak{g}=\mathfrak{so} (V_1) \oplus \mathfrak{so}(V_2),
$$  
unless $dim \bigskip V_1 = dim \bigskip V_2 =2$.
\end{lem}
\begin{proof}
Assume first that $dim V_1\neq dim V_2$. We will prove the case where $dim V_1 < dim V_2 $, from which the case $dim V_2 < dim V_1 $ follows. 
Let $\{X_i\} _{i=1,...,n_1}$ be an orthonormal basis for $V_1$. Since the $2$-form $F$ vanishes over  $V_i\wedge V_i$, $i=1,2$, we can express it as
$$
F=X_1\wedge Y_1 +...+X_{n_1}\wedge Y_{n_1}
$$
where $Y_1,...,Y_{n_1} \in V_2$ are not all equal to zero. 
Then from (\ref{eq}) it follows that $\mathfrak{g}$ is a subalgebra of $\mathfrak{so} (V_1) \oplus \mathfrak{so}(V_2)$, because for $2$-forms $(X \wedge Y)$, with $X\in V_1$, $Y \in V_2$,

$$
[F,(X \wedge Y)]=\sum_{k=1}^{n_1} \langle Y_k,Y \rangle X_k\wedge X + \langle X_k,X \rangle Y_k\wedge Y.
$$
Now it suffices to show that $\mathfrak{g}$ contains $\mathfrak{so} (V_1) \oplus\{0\}$ and $\{0\} \oplus \mathfrak{so}(V_2)$.
For this we consider the orthogonal complement of the space generated by the $Y_1,...,Y_{n_1}$ in $V_2$, such that 
$$
V_2= Vect(Y_1,...,Y_{n_1})\oplus  Vect(Y_1,...,Y_{n_1})^{\bot}.
$$
The orthogonal complement $Vect(Y_1,...,Y_{n_1})^{\bot}$ is non-zero since by assumption $dim(V_1)<dim(V_2)$.
Then for all $Y \in Vect(Y_1,...,Y_{n_1})^{\bot}$, since $Y_i\bot Y$ 
$$
[F,(X_i \wedge Y)]= [X_i \wedge Y_j,(X_i \wedge Y)]= Y_i \wedge Y \in \mathfrak{g}.
$$
For all non-zero $Y \in Vect(Y_1,...,Y_{n_1})^{\bot}$ and $X_i,X_j \in V_1$,
\begin{align*}
\frac{1}{|Y|^2} [[F,(X_i \wedge Y ) ],[F,(X_j \wedge Y)]]&= \frac{1}{|Y|^2} [Y_i \wedge Y,Y_j \wedge Y] \\ &= Y_i \wedge Y_j \in \mathfrak{g}.
\end{align*}
Also, for all $Y,Y' \in Vect(Y_1,...,Y_{n_1})^{\bot}$ and $k \leqslant n_1$, such that $Y_k \neq 0$
\begin{align*}
\frac{1}{|Y_{k}|^2} [[F,(X_k \wedge Y)],[F,(X_k \wedge Y')]] &= \frac{1}{|Y_{k}|^2} [Y_k \wedge Y,Y_k \wedge Y'] \\ &= Y \wedge Y' \in \mathfrak{g}.
\end{align*}
Hence $\{0\} \oplus \mathfrak{so} (V_2)\subset \mathfrak{g}$. \\
Now, we consider an orthonormal basis $Y'_1,...,Y'_{n_2}$ for $V_2$ and express $F$ in this basis 
$$
F=X'_1\wedge Y'_1 +...+X'_{n_2}\wedge Y'_{n_2}
$$
with $X'_1,...,X'_{n_2} \in V_1$ not all equal to zero. If $X'_1,...,X'_{n_2}$ generate $V_1$, we can generate $\mathfrak{so} (V_1) \oplus\{0\}$ in $\mathfrak{g}$, with 
$$
[F,X'_i\wedge Y'_j]-\sum_{k=1}^{n_2} \langle X'_k,X'_i \rangle Y'_k\wedge Y'_j = X'_j\wedge X'_i \in \mathfrak{g},
$$
since we already showed that $\sum_{k=1}^{n_2} \langle X'_k,X'_i \rangle Y'_k\wedge Y'_j \in \mathfrak{g}$ . 
If $X'_1,...,X'_{n_2}$ do not generate $V_1$, we can split $V_1$,
$$
V_1=Vect(X'_1,...X'_{n_2}) \oplus Vect(X'_1,...X'_{n_2})^{\bot}, 
$$
and analogous to above we can show that $\mathfrak{so} (V_1) \oplus\{0\}\subset \mathfrak{g}$, hence $\mathfrak{g}=\mathfrak{so} (V_1) \oplus \mathfrak{so} (V_2)$. \\
If $dim V_1 =dim V_2 \geqslant 3$, we express as before the $2$-form $F$ with respect to an orthonormal basis $X_1,...,X_{n_1}$ of $V_1$, 
$$
F=X_1\wedge Y_1 +...+X_{n_1}\wedge Y_{n_1}
$$
for some $Y_1,...,Y_{n_1} \in V_2$ not all equal to zero. 
If the vectors $Y_1,...,Y_{n_1}$ are linearly dependent we proceed as we did in the case $dim V_1 \neq dim V_2$. 
If the vectors $Y_1,...,Y_{n_1}$ are linearly independent, consider for all $k  \in \{1,...,dim(V_2)\}$, $Y'_k \in V_2$, such that $Y'_k \bot Y_i$ for all $i\neq k$ and $ \langle Y_k,Y'_k \rangle =1$. Then from (\ref{eq})
\begin{align*}
[F,X_i \wedge Y'_j]& =[X_i\wedge Y_i, X_i\wedge Y'_j]+[X_j\wedge Y_j, X_i\wedge Y'_j]\\ & = Y_i\wedge Y'_j + X_j \wedge X_i.
\end{align*}   \\
Taking $i,j,k \in \{ 1, ...,dim V_1  \}$ all different, it yields,
\begin{align*}
\frac{1}{(|Y_i|^2+|Y_j|^2)} &([[F,X_i\wedge Y'_j],[F,X_i \wedge Y'_k]]+[F,X_j\wedge Y'_k]) + [[F,X_j \wedge Y'_j],[F,X_j \wedge Y'_k]]) \\
& = \frac{1}{(|Y_i|^2+|Y_j|^2)} ([X_j\wedge X_i, X_k \wedge X_i]+[Y_i\wedge Y'_j,Y_i \wedge Y'_k]+X_k\wedge X_j \\ 
& + Y_j \wedge Y'_k +[Y_j\wedge Y'_j , Y_j \wedge Y'_k]) \\
& = \frac{1}{(|Y_i|^2+|Y_j|^2)} (X_j\wedge X_k+|Y_i|^2Y'_j\wedge Y'_k + X_k\wedge X_j+ Y_j \wedge Y'_k \\ 
& + |Y_j|^2 Y'_j \wedge
Y'_k - Y_j \wedge Y'_k) \\
& =Y'_j \wedge Y'_k \in \mathfrak{g}.
\end{align*}
Thus as before we can prove that $\mathfrak{g}=\mathfrak{so} (V_1) \oplus \mathfrak{so}(V_2)$.
\end{proof}
Now we can prove our main result.
\begin{proof}[Proof of Theorem \ref{thm}] 
Recall the Ambrose-Singer theorem, which says that $$span\{R^D (X\wedge Y)|X,Y \in TM\} \subset  \mathfrak{hol} (D).$$
Since $CO(n_1)+CO(n_2)$ is not in $CO(n_1+n_2)$, the "generic" case here is 
$$
(CO(n_1 )+CO(n_2 ) )\cap CO(n_1 +n_2 ) = \R_{ + } ^{ \ast } \times SO(n_1) \times SO(n_2).
$$
Hence it suffices to show that 
$
span\{R^D(X \wedge Y)| X,Y \in TM\}= \R  \oplus \mathfrak{so} (n_1) \oplus \mathfrak{so}(n_2).
$ 
For a conformal product, for all $X \in TM $ and $Y \in TM_i$, $i=1,2$ lifted to a vector field on $M$, it results that $
D_X Y$ is in  $TM_i$.
This means for the curvature tensor $R^D$, for $X \in TM_1$, $Y \in TM_2$ lifted to vector fields on $M$,
$
R^D (Z,T,X,Y)=c(R^D_{Z,T}X,Y)=0
$
for all $Z,T \in TM$, and
$
c(X,Y)=F(Z,T)c(X,Y)=0.
$
Thus considering the symmetry failure of $R^D$ in (\ref{RofD}), measured by the Faraday form, for $X \in TM_1$ and $Y \in TM_2$
\begin{align*}
R^D (X,Y,Z,T) =&R^D(Z,T,X,Y) + (F(X)\wedge Y - F(Y)\wedge X )(Z,T) \\  &+F(X,Y)c(Z,T)- F(Z,T)c(X,Y)\\ =& (F(X)\wedge Y - F(Y)\wedge X )(Z,T)+F(X,Y)c(Z,T),
\end{align*}
hence
\[
R^D (X , Y)= (F\wedge Id)(X,Y) + F(X,Y)Id = [F^{\sharp},X \wedge Y^{\flat}] + F(X,Y)Id,
\]
where the bracket $ [ \cdot , \cdot  ] $ denotes the commutator in End(TM) i.e. $R^D(\omega )= [F^{\sharp },\omega  ^{\flat}]+ c(F,\omega )Id$ for all $\omega \in T^{\ast } M_1 \otimes T^{\ast } M_2 \subset \Lambda ^2 M$.
Since $R^D(\omega ) \in \mathfrak{hol} (D)$ for all $\omega \in  \Lambda ^2 M$, we get in particular 
$$
[F^{\sharp},\omega ^{\flat}]+ c(F,\omega )Id  \in \mathfrak{hol} (D),
$$
for every two-form $\omega \in T^{\ast } M_1 \otimes T^{\ast } M_2$.
Furthermore we remark that the Faraday form is element of $T^{\ast } M_1 \otimes T^{\ast } M_2 $, hence for $\omega := F$,
$$
R^D (F)= \langle F,F \rangle Id + [F^{\sharp},F^{\flat}]= ||F ||^2Id,
$$
we obtain $\R \subset\mathfrak{hol} (D) $, since $F$ is not identically zero, i.e. in the non-closed case 
\[
span\{R(F)\}= \R
\]
generates the dilation in the holonomy group. \\
Then by the algebra properties, 
$$
R^D (T^{\ast } M_1 \otimes T^{\ast } M_2)- F(T^{\ast } M_1 , T^{\ast } M_2)Id =[F,T^{\ast } M_1 \otimes T^{\ast } M_2] \in \mathfrak{hol} (M,D),
$$ 
and using Lemma \ref{lem} it follows directly, that $[F,T^{\ast } M_1 \otimes T^{\ast } M_2]$ generates $\mathfrak{so} (n_1) \oplus \mathfrak{so}(n_2)$ in $\mathfrak{hol} (D)$. Thus
$
span\{R^D(X \wedge Y)| X,Y \in TM\}= \R  \oplus \mathfrak{so} (n_1) \oplus \mathfrak{so}(n_2),
$
which shows that the holonomy group is 
$$
Hol_0(D)= \R_{ + } ^{ \ast } \times SO(n_1) \times SO(n_2),
$$
if we don't have $n_1 = n_2 =1$ or $2$. \\
If $n_1=n_2=2$, from calculation in \cite{BeMo} it follows that the holonomy has at least dimension 2, which leaves us with two possibilities, i.e. the generic case and  $\C ^{\ast } \simeq  \R_{ + } ^{ \ast } \times SO(2) \subset CO(4)$. \\
In \cite{BeMo} it was shown that \em hyper-Hermitian \em non-closed conformal products have restricted holonomy equal to $\C ^{\ast }$.
A hyper-Hermitian non-closed product is a non-closed conformal \em multi-product \em, which is a product structure, whose adapted Weyl structure $D$ leaves parallel more than one pair of complementary distributions.

For the sake of completeness, if $n_1=n_2 =1$, then  $
span\{R(F)\}= \R
$, hence
\[
Hol_0 (D)= \R_{ + } ^{ \ast } =  \R_{ + } ^{ \ast } \times SO(1) \times SO(1).
\]

\end{proof}

\section{Einstein-Weyl conformal products}
A Weyl manifold $(M,c,D)$ is called \em Einstein-Weyl \em if the trace-free symmetric part of the Ricci tensor $Ric^D$ vanishes.
In this last section we want to prove that there are no non-closed Einstein-Weyl product structures $(M,c,D)$ except when $n=4$ by generalizing a proposition in \cite{BeMo}.
In particular in  \cite{BeMo} it was given a local characterization of all Einstein-Weyl structures $(M,c,D)$ in dimension $4$ with reducible holonomy.    
As before, we denote by $F$ the Faraday form which is obtained by extending a section $F_0$ of $T^{\ast } M_1 \otimes T^{\ast } M_2 $ to a skew-symmetric bilinear form on $TM=TM_1\oplus TM_2$. We can extend $F_0$ to a symmetric bilinear form $\hat{F}$, for all $X_i,Y_i \in TM_i$:  $$\hat{F}(X_i,Y_i):= 0; \ \ \hat{F}(X_1,X_2):=F(X_1,X_2) \ \ \hat{F}(X_1,X_2)=-F(X_2,X_1). $$ 
\begin{satz}
A non-closed conformal product $(M, c , D )$ with $M=M^{n_1}_1 \times M^{n_2}_2$, $c=[g_1+e^{2f}g_2]$ and $D$ the adapted Weyl structure is Einstein-Weyl if and only if it is locally isomorphic to a conformal product $(M_1 \times M_2 ,c= [g_1 +e^{2f} g_2],D) $, where $M_1 $ and $M_2 $ are open sets of $\R ^2 $, $g_i$ is the flat metric on $ M_i$ and the function $f: M_1 \times M_2 \subset \R ^4 \to \R $ satisfies the $Toda$-type equation 
$$
e^{2f}(\partial _{11} f + \partial _{22} f) + \partial _{33} f + \partial
 _{44} f =0.
$$ 
\end{satz}

Let $\hat{F}$ be the symmetrization of the Faraday form and $Ric^i$ the Ricci curvature of $M_i$, $i=1,2$ with respect to the metric $h_i:=e^{\epsilon (i) 2f}g_i$, with $\epsilon(i):=(-1)^i$. Then we can write the Ricci curvature $Ric ^D$ of $M$ with respect to $D$ \cite{BeMo}, as
$$
Ric^D = Ric^1 + Ric^2 + \frac{2-n}{2}F + \frac{n_1-n_2}{2} \hat{F},
$$ 
which tells us that in the non-closed case $(M,c,D)$ is Einstein-Weyl if and only if $n_1=n_2$ and 
\begin{align} \label{Ric}
Ric^1 +Ric^2 =\varphi(g_1 +e^{2f}g_2)
\end{align}
for some function $\varphi :M \to \R$. 

If $n_1=n_2 =2 $, as shown in \cite{BeMo}, we know that every 2-dimensional metric is locally conformal to the flat metric, hence without loss of generality we can assume that $g_1,g_2$ are flat. Using the formula for the conformal change of the Ricci tensor (\cite{Besse}, 1.159d), for the metrics $g_i,e^{\epsilon (i) 2f}g_i$, $i=1,2$ we have 
\begin{align*}
Ric^1 &=Ric^{g_1}-(n-2)(Ddf-df\circ df))+ (\Delta _1 f -(n-2)|df|^2)g_1 \\
&= \Delta _1 f g_1
\end{align*}
and 
$$
Ric^2 =-\Delta _2 f g_2, 
$$
where $\Delta _i$ denotes the partial Laplacian on $\R^4 =\R ^2 \times \R ^2$, $\Delta _1 f = \partial _{11}f + \partial _{22}f$ and $\Delta _2 f = \partial _{33}f + \partial _{44}f$.
Thus 
$$
-(\Delta _1 f)g_1 + (\Delta _2)g_2 =(-\partial _{11}f - \partial _{22}f)g_1 + (\partial _{33}f + \partial _{44}f) g_2 = \varphi(g_1 +e^{2f}g_2),
$$
and 
$$
-\partial _{11}f - \partial _{22}f= \varphi = e^{-2f} (\partial _{33}f + \partial _{44}f)
$$
which gives us the equation
$$
e^{2f}(\partial _{11} f + \partial _{22} f) + \partial _{33} f + \partial
 _{44} f =0.
$$ 

Assume now $n_1=n_2 \geqslant 3$. By (\ref{Ric}) $(M,c,D)$ is Einstein-Weyl if 
$$
Ric^1 +Ric^2 =\varphi(g_1 +e^{2f}g_2)
$$
for some function $\varphi :M \to \R$. 
Then
$$
Ric^1 =\varphi g_1 = \varphi e^{2f} (e^{-2f}g_1)= \varphi e^{2f} h_1
$$
and
$$
Ric^2=\varphi e^{2f}g_2 =\varphi h_2,
$$
hence $M_1,M_2$ are both Einstein with respect to the metrics $h_1$ and $h_2$ respectively. A classical result for Einstein manifolds says that for a manifold $(M^n,g)$, if $Ric^g =\psi g$, 
then $\psi $ is constant if $n \geqslant 3$.
This implies that for all $x_2 \in M_2 $, $\varphi e^{2f}(\cdot ,x_2)$ is equal to a constant $c_1(x_2)$ and for all $x_1 \in M_1 $, $\varphi (x_1,\cdot )$ is equal to a constant $c_2(x_1)>0$. Thus for all    
$(x_1,x_2)\in M_1\times M_2$: 
$$
c_2 (x_1)e^{2f(x_1,x_2)}=c_1(x_2)> 0,
$$      
meaning that $f(x_1,x_2)=\frac{1}{2} ln(c_1(x_2))-\frac{1}{2} ln(c_2(x_1))$ is a sum of two functions only depending on one factor, which is only possible if $(M,c,D)$ is a Riemannian product and $F$ is closed.
Hence there are no non-closed conformal product manifolds with $n_1=n_2 \geqslant 3$. \\

\vspace{0.5cm}
Universit\'e de Versailles Saint-Quentin, Laboratoire de Math\' ematiques, 45 avenue des Etats-Unis, 78035 Versailles, France \\
$E$-$mail$ $address :  \mathtt{jonas.grabbe@uvsq.fr}$
\end{document}